\documentclass[10pt,a4paper]{amsart}
\usepackage[T1]{fontenc}
\usepackage[utf8]{inputenc}
\usepackage[english]{babel}
\usepackage{amsmath, amssymb, amsthm}
\usepackage{cmap}
\theoremstyle{plain}
\newtheorem{theorem}{Theorem}
\newtheorem{lemma}{Lemma}
\theoremstyle{remark}
\newtheorem*{remark}{Remark}

\title{The power set of a quasinilpotent backward weighted shift}
\author{E.~Ignatev}
\address{Faculty of Mathematics and Mechanics, St.~Petersburg State University, St.~Petersburg, Russia}
\email{ignatevgeorge@gmail.com}
\date{July 17, 2026}

\begin{document}

\begin{abstract}
For a quasinilpotent operator $T$ on a Banach space $X$, R.~Douglas and R.~Yang associated with each nonzero vector $x$ the local resolvent-growth exponent $k_x$, and introduced the power set $\Lambda(T) = \{k_x : x \neq 0\}$. We prove that $1 \in \Lambda(T)$ for every quasinilpotent operator on an arbitrary Banach space, which answers a question of Ji and Zhang. We further show that $\Lambda(T) = [0,1]$ for every backward unilateral weighted shift on $\ell^p$ whose weight sequence is strictly decreasing and $p'$-summable for some $p' > 0$, thereby weakening the hypotheses imposed by Hu and Ji.
\end{abstract}

\maketitle

\section{Introduction and notation}
Let $X$ be a Banach space and let $T \in \mathcal{L}(X)$ be a quasinilpotent operator, that is, $\sigma(T) = \{0\}$. Since $\sigma(T) = \{0\}$, the resolvent
\[
R(\lambda) := (\lambda - T)^{-1} \in \mathcal{L}(X)
\]
is defined for every $\lambda \in \mathbb{C} \setminus \{0\}$. For a nonzero $x \in X$, set
\[
k_x := \limsup_{z \to 0}\frac{\ln\|(z - T)^{-1}x\|}{\ln\|(z - T)^{-1}\|}, \qquad
\Lambda(T) := \{k_x : x \in X,\, x \neq 0\}.
\]
The set $\Lambda(T)$ was introduced by R.~Douglas and R.~Yang in~\cite{DouY21}; a standard verification shows that $\Lambda(T) \subseteq [0,1]$. A nonempty subset $X \subseteq \mathbb{R}$ is called \emph{right-closed} if $\sup Y \in X$ for every nonempty bounded subset $Y \subseteq X$.
In~\cite{Ji-Zhang}, qualitative properties of the set $\Lambda(T)$ are investigated for an arbitrary quasinilpotent operator; in particular, the following results are proved.
\begin{theorem}[\cite{Ji-Zhang}]
The set $\Lambda(T)$ is right-closed.
\end{theorem}
\begin{theorem}[\cite{Ji-Zhang}]
For every right-closed subset $\sigma \subseteq [0,1]$ containing $1$, there exists a quasinilpotent operator $T \in \mathcal{L}(H)$ such that $\Lambda(T) = \sigma$.
\end{theorem}
The question of whether $1 \in \Lambda(T)$ always holds was also raised in~\cite{Ji-Zhang}. In Section~\ref{sec:one} of the present paper we give an affirmative answer.
In~\cite{Hu-Ji}, the power set of certain quasinilpotent weighted shifts on $\ell^p$ is studied. In particular, for backward unilateral weighted shifts under strong restrictions on the weights it is shown that $\Lambda(T) = [0,1]$. In Section~\ref{sec:shift} of the present paper, the conditions on the weights are substantially weakened, in a certain sense, while the equality $\Lambda(T) = [0,1]$ is retained.
\section{$1 \in \Lambda(T)$ for a quasinilpotent operator}\label{sec:one}
\begin{theorem}\label{thm:one}
Let $X$ be a Banach space and let $T \in \mathcal{L}(X)$ satisfy $\sigma(T) = \{0\}$. Then there exists $x \in X \setminus \{0\}$ such that $k_x = 1$. In particular, $1 \in \Lambda(T)$.
\end{theorem}
\begin{proof}
Suppose, for contradiction, that $k_x < 1$ for every $x \neq 0$.
For each integer $n \geq 2$, define
\[
E_{n}
:=
\left\{
x \in X \; \middle| \;
\|R(\lambda)x\|
\le
\|R(\lambda)\|^{1-\frac1n}
\ \text{for all } \lambda \text{ with }
0 < |\lambda| < \frac1n
\right\}.
\]
1) The set $E_{n}$ is closed. Indeed,
for a fixed $\lambda$ the map
\[
x \mapsto \|R(\lambda)x\|
\]
is continuous, since $R(\lambda) \in \mathcal{L}(X)$.
The quantity $\|R(\lambda)\|^{1-1/n}$ is a constant.
Hence the set
\[
\{x : \|R(\lambda)x\| \le \|R(\lambda)\|^{1-1/n}\}
\]
is closed for each $\lambda$. Consequently,
$E_{n}$ is an intersection of closed sets
and is therefore closed.
2) We show that, under our hypothesis,
\[
X = \bigcup_{n \ge 2} E_{n}.
\]
Let $x \neq 0$.
By assumption $k_x < 1$, so there exists $c$ such that
\[
\limsup_{\lambda \to 0}
\frac{\ln \|R(\lambda)x\|}{\ln \|R(\lambda)\|}
< c < 1.
\]
This means that there exists $\delta > 0$ such that
for all $0 < |\lambda| < \delta$
\[
\frac{\ln \|R(\lambda)x\|}{\ln \|R(\lambda)\|}
\le c.
\]
Note that for $|\lambda| < 1$ one has
$\|R(\lambda)\| \ge 1/|\lambda| > 1$, and therefore
$\ln \|R(\lambda)\| > 0$;
the previous inequality then yields
$\|R(\lambda)x\| \le \|R(\lambda)\|^c$.
Choose $n$ large enough so that simultaneously
\[
1-\frac1n \ge c
\quad \text{and} \quad
\frac1n < \delta.
\]
This is possible for $n$ sufficiently large.
Then for $0 < |\lambda| < 1/n$ we have $|\lambda| < \delta$
and $\|R(\lambda)\| > 1$,
so that
\[
\|R(\lambda)x\|
\le
\|R(\lambda)\|^c
\le
\|R(\lambda)\|^{1-\frac1n},
\]
that is, $x \in E_{n}$. Trivially, $0 \in E_{n}$ as well.
3) Since $X$ is a Banach space, by the Baire category theorem $X$ cannot be a countable union of nowhere dense closed sets. Since
\[
X = \bigcup_{n \ge 1} E_{n}
\]
and every $E_{n}$ is closed,
some $E_{n_0}$
has nonempty interior.
Consequently, there exist
$x^* \in X$, $r > 0$, and $n_0 \ge 1$ such that
\[
B(x^*,r) \subset E_{n_0}.
\]
This means that for every $y$ with
$\|y-x^*\| \le r$
and every $0 < |\lambda| < 1/n_0$,
\[
\|R(\lambda)y\|
\le
\|R(\lambda)\|^{1-\frac1{n_0}}.
\tag{1}
\]
Let $0 < |\lambda| < 1/n_0$.
For any $z \in X$ with $\|z\| \le 1$,
the vectors
\[
x^* + rz
\quad \text{and} \quad
x^*
\]
both lie in $B(x^*,r)$,
so by (1)
\[
\|R(\lambda)(x^* + rz)\|
\le
\|R(\lambda)\|^{1-\frac1{n_0}},\,\,\,
\|R(\lambda)x^*\|
\le
\|R(\lambda)\|^{1-\frac1{n_0}}.
\]
By linearity of $R(\lambda)$,
\[
r\|R(\lambda)z\|
=
\|R(\lambda)(rz)\|
\le
\|R(\lambda)(x^* + rz)\|
+
\|R(\lambda)x^*\|
\le
2\|R(\lambda)\|^{1-\frac1{n_0}}.
\]
Taking the supremum over $\|z\| \le 1$,
\[
\|R(\lambda)\|
\le
\frac{2}{r}\|R(\lambda)\|^{1-\frac1{n_0}}.
\]
Hence
\[
\|R(\lambda)\|^{1/n_0}
\le
\frac{2}{r},
\]
so that
\[
\|R(\lambda)\|
\le
\left(\frac{2}{r}\right)^{n_0}
\]
for all $0 < |\lambda| < 1/n_0$.
Thus $\|R(\lambda)\|$
is bounded on a punctured neighborhood of $0$,
which is a contradiction.
\end{proof}
\begin{remark}
The question of whether $1$ belongs to $\Lambda(T)$ was posed in~\cite{Ji-Zhang} in the Hilbert space setting; Theorem~\ref{thm:one} answers it in an arbitrary Banach space.
\end{remark}
\section{Backward weighted shift on $\ell^p$}\label{sec:shift}
Throughout this section $\mathbb{N} = \{0, 1, 2, \dots\}$, the vectors $\{e_n\}_{n \geq 0}$ form the canonical basis of $\ell^p(\mathbb{N})$, and $f_m$ denotes the $m$-th coordinate functional on $\ell^p(\mathbb{N})$. Given a sequence $\{w_n\}_{n = 1}^{\infty}$ of positive numbers, the \emph{backward unilateral weighted shift} $T$ on $\ell^p(\mathbb{N})$ with weight sequence $\{w_n\}$ is the operator determined by
\[
T e_0 := 0, \qquad T e_n := w_n e_{n - 1} \quad (n \geq 1).
\]
The following result is proved in~\cite{Hu-Ji}: if $T$ is a backward unilateral weighted shift on $\ell^p(\mathbb{N})$ with weight sequence $\{w_n\}_{n=1}^{\infty}$ satisfying $\frac{1}{n + 2 m_0} \leq w_{n + m_0} \leq \frac{1}{n}$ for $n \geq 1$ and some $m_0 \geq 1$, then $\Lambda(T) = [0,1]$.
The proof in~\cite{Hu-Ji} relies on two technical ingredients: (i)~the special structure of the test vector $x_r = \{a_m\}_{m \geq 0}$, whose coordinates satisfy the coordinatewise majorization $|f_m(T^k x_r)| \leq a_m \cdot f_0(T^k x_r)$; (ii)~the identity $\|T^k\| = \prod_{j=1}^{k} w_j$, which allows one to match the numerator estimate with the asymptotics of $\ln\|(z - T)^{-1}\|$. In the present paper we construct a test vector of a different form, based on the convexity of the function $\Phi(k) = -\ln\prod_{j=1}^{k} w_j$, and obtain the same result under the condition of $p'$-summability of the weights.
\begin{theorem}\label{thm:main}
Let $1 \leq p < \infty$, and let $T$ be a backward unilateral weighted shift on $\ell^p(\mathbb{N})$ with strictly decreasing weight sequence $\{w_n\}_{n=1}^{\infty}$ satisfying $\sum_{n=1}^{\infty} w_n^{p'} < \infty$ for some $p' > 0$. Then $\Lambda(T) = [0, 1]$.
\end{theorem}
Before the proof, we state a technical lemma which will be applied twice---to the asymptotics of the numerator and of the denominator in the definition of $k_x$.
\begin{lemma}\label{lem:laplace}
Let $\{\psi(k)\}_{k \geq 0}$ satisfy $\psi(0) = 0$, suppose the difference sequence $\tilde\phi_k := \psi(k) - \psi(k - 1)$ is strictly increasing with $\tilde\phi_k \to +\infty$, and assume that
\begin{equation}\label{eq:phi-log}
\tilde\phi_k \geq \tilde c \ln k, \qquad k \geq k_0,
\end{equation}
for some $\tilde c > 0$ and $k_0 \in \mathbb{N}$. Set
\[
S(\mu) := \sum_{k \geq 0} e^{k\mu - \psi(k)}, \qquad M(\mu) := \max_{k \geq 0}\bigl(k\mu - \psi(k)\bigr).
\]
Then $\ln S(\mu) = M(\mu)(1 + o(1))$ as $\mu \to +\infty$.
\end{lemma}
\begin{proof}
Set $h(k) := k\mu - \psi(k)$. The differences $h(k+1) - h(k) = \mu - \tilde\phi_{k+1}$ are strictly decreasing in $k$; hence $h$ is strictly concave on $\mathbb{Z}_{\geq 0}$ and has a unique maximizer $k^* = k^*(\mu)$, characterized by
\begin{equation}\label{eq:kstar}
\tilde\phi_{k^*} \leq \mu < \tilde\phi_{k^* + 1}.
\end{equation}
From~\eqref{eq:phi-log} and~\eqref{eq:kstar} we infer that $\ln k^* \leq \mu / \tilde c$ and that $k^* \to \infty$ as $\mu \to \infty$.
The lower estimate is trivial: $\ln S(\mu) \geq h(k^*) = M(\mu)$. For the upper estimate, write
\[
R(\mu) := \sum_{k \geq 0} e^{h(k) - h(k^*)}, \qquad S(\mu) = e^{M(\mu)} R(\mu),
\]
and estimate $\ln R(\mu)$. Since $k^*$ maximizes $h$, the differences $h(k) - h(k^*)$ are nonpositive; in particular the summands with $0 \leq k \leq k^*$ are each at most $1$, and there are $k^* + 1$ of them. For $j \geq 1$, telescoping gives
\[
h(k^*) - h(k^* + j) = \sum_{i = 1}^{j}(\tilde\phi_{k^* + i} - \mu),
\]
where the summands are nonnegative by~\eqref{eq:kstar} and increasing in $i$. Set
\[
j_0 := \min\{j \geq 1 : \tilde\phi_{k^* + j} - \mu \geq 1\};
\]
the existence of $j_0$ is guaranteed by $\tilde\phi_k \to \infty$, and~\eqref{eq:phi-log} yields $j_0 \leq e^{(\mu + 1)/\tilde c}$. For $j \geq j_0$ we have $h(k^*) - h(k^* + j) \geq j - j_0 + 1$, so that
\[
\sum_{j \geq j_0} e^{h(k^* + j) - h(k^*)} \leq \sum_{j \geq j_0} e^{-(j - j_0 + 1)} = \frac{1}{e - 1}.
\]
The summands with $1 \leq j < j_0$ are each at most $1$, and their number is less than $j_0$. Consequently,
\[
R(\mu) \leq k^* + j_0 + \frac{e}{e - 1}, \qquad \ln R(\mu) \leq \frac{\mu + 1}{\tilde c} + O(1) = O(\mu).
\]
For any fixed $N \in \mathbb{N}$, we have $M(\mu) \geq N\mu - \psi(N)$, whence $M(\mu)/\mu \to +\infty$. Therefore $\ln R(\mu) = o(M(\mu))$, and
\[
\ln S(\mu) = M(\mu) + \ln R(\mu) = M(\mu)(1 + o(1)). \qedhere
\]
\end{proof}
\begin{proof}[Proof of Theorem~\ref{thm:main}]
Set
\[
W_0 := 1,\quad W_k := \prod_{j = 1}^{k} w_j,\quad \phi_j := -\ln w_j,\quad \Phi(k) := -\ln W_k = \sum_{j = 1}^{k} \phi_j.
\]
The action of $T^k$ on the canonical basis is given by $T^k e_n = (W_n / W_{n - k}) e_{n - k}$ for $n \geq k$ and $T^k e_n = 0$ for $n < k$. Since $\{w_n\}$ is decreasing, $\sup_{n \geq k} W_n / W_{n - k} = W_k$, and consequently $\|T^k\| = W_k$, while $\phi_j$ is strictly increasing with $\phi_j \to +\infty$. Extend $\Phi$ to $[0, +\infty)$ piecewise linearly; since the slopes $\phi_{k+1}$ are strictly increasing, the resulting function is convex, and its right derivative at an integer point $k$ equals $\phi_{k + 1}$.
Since $\Lambda(T) \subseteq [0, 1]$ and $1 \in \Lambda(T)$ by Theorem~\ref{thm:one}, it suffices to show that $0 \in \Lambda(T)$ and that, for every $\tau \in (0, 1)$, there exists $x_\tau \in \ell^p$ with $k_{x_\tau} = \tau$.
The inclusion $0 \in \Lambda(T)$ is elementary: $(z - T)^{-1} e_m = \sum_{j = 0}^{m} T^j e_m / z^{j + 1}$ is a polynomial in $1/z$ of degree $m + 1$, so $\ln\|(z - T)^{-1} e_m\|_p = O(\ln(1/|z|))$, whereas $\ln\|(z - T)^{-1}\|$ grows faster than any linear function of $\ln(1/|z|)$ (see~\eqref{eq:resolvent-norm} below). Hence $k_{e_m} = 0$.
Fix $\tau \in (0, 1)$ and set
\begin{equation}\label{eq:xtau}
\Phi_\tau(k) := \tau\, \Phi(k/\tau), \qquad x_\tau := \{a_k\}_{k \geq 0},\quad a_k := e^{\Phi(k) - \Phi_\tau(k)}.
\end{equation}
By convexity of $\Phi$ together with $\Phi(0) = 0$, $\Phi_\tau(k) \geq \Phi(k)$, so $0 < a_k \leq 1$ and
\begin{equation}\label{eq:akWk}
a_k W_k = e^{-\Phi_\tau(k)}.
\end{equation}
\smallskip
\emph{$x_\tau \in \ell^p$.} By convexity of $\Phi$, for $t \geq 0$ we have $\Phi(t) \geq \Phi(k) + \phi_{k + 1}(t - k)$. Substituting $t = k/\tau$ and multiplying by $\tau$ yields
\begin{equation}\label{eq:phitau-phi}
\Phi_\tau(k) - \Phi(k) \geq (1 - \tau)\bigl(k \phi_{k + 1} - \Phi(k)\bigr) = (1 - \tau)\sum_{j = 1}^{k}(\phi_{k + 1} - \phi_j).
\end{equation}
Convergence of $\sum w_n^{p'}$ together with the monotonicity of $\phi_n$ implies $n e^{-p' \phi_n} \to 0$, whence
\begin{equation}\label{eq:phi-bound}
\phi_n \geq c \ln n \qquad (n \geq k_1)
\end{equation}
for some $c > 0$ and $k_1 \in \mathbb{N}$. Retaining only the first $k_1$ terms in~\eqref{eq:phitau-phi} and applying~\eqref{eq:phi-bound},
\[
\Phi_\tau(k) - \Phi(k) \geq (1 - \tau)(k_1 c \ln k - S_0), \qquad S_0 := \sum_{j = 1}^{k_1} \phi_j,
\]
so that $a_k^p \leq e^{p(1 - \tau) S_0}\, k^{-p(1 - \tau) k_1 c}$. By increasing $k_1$, the exponent $p(1 - \tau) k_1 c$ can be made arbitrarily large; in particular, if $p(1 - \tau) k_1 c \geq 2$, then $\sum a_k^p < \infty$.
\smallskip
\emph{Coordinatewise majorization.} We show that for all $m, k \geq 0$,
\begin{equation}\label{eq:coord-major}
|f_m(T^k x_\tau)| \leq a_m \cdot f_0(T^k x_\tau).
\end{equation}
The operator $T^k$ sends the coordinate with index $n$ to position $n - k$ with multiplier $W_n / W_{n - k}$, so
\[
f_m(T^k x_\tau) = a_{m + k} \cdot \frac{W_{m + k}}{W_m}, \qquad f_0(T^k x_\tau) = a_k \cdot W_k,
\]
and applying~\eqref{eq:akWk},
\begin{equation}\label{eq:ratio}
\frac{f_m(T^k x_\tau)}{f_0(T^k x_\tau)} = \frac{a_{m + k} W_{m + k}}{a_k W_k W_m} = \frac{e^{-\Phi_\tau(m + k)}}{e^{-\Phi_\tau(k)} \cdot W_m} = \frac{e^{\Phi_\tau(k) - \Phi_\tau(m + k)}}{W_m}.
\end{equation}
The function $\Phi_\tau$ is convex with $\Phi_\tau(0) = 0$. Applying convexity at the points $0$ and $m + k$ with weights $\tfrac{k}{m + k}$ and $\tfrac{m}{m + k}$, we obtain
\[
\Phi_\tau(m) \leq \frac{m}{m + k} \Phi_\tau(m + k), \qquad \Phi_\tau(k) \leq \frac{k}{m + k} \Phi_\tau(m + k),
\]
and adding these inequalities gives $\Phi_\tau(m + k) \geq \Phi_\tau(m) + \Phi_\tau(k)$, whence $\Phi_\tau(k) - \Phi_\tau(m + k) \leq -\Phi_\tau(m)$. Substituting into~\eqref{eq:ratio},
\[
\frac{f_m(T^k x_\tau)}{f_0(T^k x_\tau)} \leq \frac{e^{-\Phi_\tau(m)}}{W_m} = e^{\Phi(m) - \Phi_\tau(m)} = a_m,
\]
which is~\eqref{eq:coord-major}. From~\eqref{eq:coord-major} and the triangle inequality, for $z \neq 0$,
\[
\bigl|f_m\bigl((z - T)^{-1} x_\tau\bigr)\bigr| = \left|\sum_{k \geq 0} \frac{f_m(T^k x_\tau)}{z^{k + 1}}\right| \leq a_m \sum_{k \geq 0} \frac{f_0(T^k x_\tau)}{|z|^{k + 1}} = a_m \cdot f_0\bigl((|z| - T)^{-1} x_\tau\bigr),
\]
where the last equality uses the positivity $f_0(T^k x_\tau) = e^{-\Phi_\tau(k)} > 0$. Raising to the $p$-th power, summing over $m$, and taking the $p$-th root,
\begin{equation}\label{eq:upper-res}
\|(z - T)^{-1} x_\tau\|_p \leq \|x_\tau\|_p \cdot f_0\bigl((|z| - T)^{-1} x_\tau\bigr).
\end{equation}
All terms in $f_0\bigl((t - T)^{-1} x_\tau\bigr) = \sum a_k W_k / t^{k + 1}$ are positive for $t > 0$, hence
\begin{equation}\label{eq:lower-res}
\|(t - T)^{-1} x_\tau\|_p \geq f_0\bigl((t - T)^{-1} x_\tau\bigr) > 0.
\end{equation}
\smallskip
\emph{Asymptotics.} Set $\mu := \ln(1/|z|)$ and $M(\mu) := \max_{k \geq 0}(k\mu - \Phi(k))$. We show that
\begin{equation}\label{eq:resolvent-norm}
\ln\|(z - T)^{-1}\| = M(\mu)(1 + o(1)).
\end{equation}
From the Neumann series, for $t > 0$,
\[
\|(t - T)^{-1}\| \leq \frac{1}{t} \sum_{k \geq 0} \frac{W_k}{t^k} = \frac{1}{t} \sum_{k \geq 0} e^{k\mu - \Phi(k)},
\]
Applying Lemma~\ref{lem:laplace} with $\psi = \Phi$, whose hypotheses hold because $\{\phi_k\}$ is strictly increasing and tends to $+\infty$ while~\eqref{eq:phi-log} follows from~\eqref{eq:phi-bound}, we obtain
\[
\ln\|(z - T)^{-1}\| \leq M(\mu)(1 + o(1)) + \mu.
\]
Conversely, for any $N \geq 0$, in the expansion $(z - T)^{-1} e_N = \sum_{k \leq N} T^k e_N / z^{k + 1}$ only the term with $k = N$ has a nonzero zeroth coordinate, so
\[
\|(z - T)^{-1}\| \geq \bigl|f_0\bigl((z - T)^{-1} e_N\bigr)\bigr| = \frac{W_N}{|z|^{N + 1}},
\]
and maximizing over $N$ gives $\ln\|(z - T)^{-1}\| \geq M(\mu)$. Since $M(\mu)/\mu \to +\infty$, we have $\mu = o(M(\mu))$, which yields~\eqref{eq:resolvent-norm}.
For the numerator, from~\eqref{eq:akWk},
\[
f_0\bigl((t - T)^{-1} x_\tau\bigr) = \frac{1}{t} \sum_{k \geq 0} e^{k\mu - \Phi_\tau(k)}, \qquad t > 0.
\]
The hypotheses of Lemma~\ref{lem:laplace} are satisfied for $\Phi_\tau$: the differences $\Phi_\tau(k) - \Phi_\tau(k - 1)$ are the average values of $\Phi'$ on intervals of length $1/\tau$, are monotonically increasing together with $\phi_j$, and satisfy~\eqref{eq:phi-log} with constant $\tilde c = c$. Applying the lemma,
\[
\ln f_0\bigl((t - T)^{-1} x_\tau\bigr) = \max_{k \geq 0}(k\mu - \Phi_\tau(k))(1 + o(1)) + \mu.
\]
The substitution $s = k/\tau$ gives
\[
\max_{k \geq 0}\bigl(k\mu - \tau \Phi(k/\tau)\bigr) = \tau \cdot \max_{s \in (1/\tau)\mathbb{Z}_{\geq 0}}(s\mu - \Phi(s)),
\]
which differs from $\tau M(\mu)$ by $O(\mu) = o(M(\mu))$. Therefore
\begin{equation}\label{eq:numerator}
\ln f_0\bigl((t - T)^{-1} x_\tau\bigr) = \tau M(\mu)(1 + o(1)).
\end{equation}
\smallskip
\emph{Computation of $k_{x_\tau}$.} From~\eqref{eq:lower-res},~\eqref{eq:numerator}, and~\eqref{eq:resolvent-norm},
\[
k_{x_\tau} \geq \limsup_{t \to 0^+}\frac{\ln f_0\bigl((t - T)^{-1} x_\tau\bigr)}{\ln\|(t - T)^{-1}\|} = \lim_{t \to 0^+}\frac{\tau M(\mu)(1 + o(1))}{M(\mu)(1 + o(1))} = \tau.
\]
On the other hand,~\eqref{eq:upper-res} implies
\[
\ln\|(z - T)^{-1} x_\tau\|_p \leq \tau M(\mu)(1 + o(1)) + \ln\|x_\tau\|_p,
\]
and the finite constant $\ln\|x_\tau\|_p$ is $o(M(\mu))$. Dividing by $\ln\|(z - T)^{-1}\| = M(\mu)(1 + o(1))$ yields $k_{x_\tau} \leq \tau$. Hence $k_{x_\tau} = \tau$.
\end{proof}
\begin{remark}
It appears that Theorem~4 remains valid without the monotonicity assumption on the weights. The corresponding generalization, or the construction of a counterexample, deserves a separate investigation.
\end{remark}
\section*{Declaration on the use of generative AI}
Generative AI (Claude Opus, Anthropic) was used substantially in the preparation of this paper. Its contribution was decisive for the formulation and proof of  Lemma~\ref{lem:laplace}, notably, and it assisted in drafting a number of the other arguments, which are standard ones. The tool was also of considerable help in preparing the \LaTeX{} source. All such output was produced under the author's direction and was subsequently revised and verified by the author. The author takes full responsibility for the contents of the paper.


\begin{thebibliography}{9}
\bibitem{DouY21}
R.~G.~Douglas, R.~Yang.
\emph{Hermitian geometry on the resolvent set (II)}.
Sci. China Math. \textbf{64} (2021), 385--398.
\bibitem{Hu-Ji}
C.~Hu, Y.~Ji.
\emph{Power set of some quasinilpotent weighted shifts on $\ell^p$}.
Linear Algebra Appl. \textbf{686} (2024), 111--133.
DOI: 10.1016/j.laa.2024.01.011.
\bibitem{Ji-Zhang}
Y.~Ji, Y.~Zhang.
\emph{On the power set of quasinilpotent operators}.
Integr. Equ. Oper. Theory \textbf{95} (2023), Article~25.
DOI: 10.1007/s00020-023-02745-4.
\end{thebibliography}
\end{document}